\documentclass{amsart}

\usepackage{amsmath,amsfonts,amssymb,amscd,verbatim,latexsym}

\usepackage{enumerate}
\usepackage{dsfont}
\usepackage{array}
\usepackage{color}
\usepackage{bbold}
\usepackage{bbm}
\usepackage{graphicx}
\usepackage[colorinlistoftodos]{todonotes}
\usepackage[colorlinks=true, allcolors=blue]{hyperref}

\renewcommand{\i}{\mathbbm{i}}
\newcommand{\mO}{\mathcal{O}}
\newcommand{\kk}{\Bbbk}

\newtheorem{theorem}{Theorem}[section]
\newtheorem{lemma}[theorem]{Lemma}
\newtheorem{proposition}[theorem]{Proposition}

\theoremstyle{definition}

\theoremstyle{remark}

\newtheorem{question}[theorem]{Question}

\numberwithin{equation}{section}

\begin{document}

\title{On the Noether bound for noncommutative rings}


\author{Luigi Ferraro}
\address{Wake Forest University, Department of Mathematics and Statistics, P. O. Box 7388, Winston-Salem, North Carolina 27109}
\email{ferrarl@wfu.edu}

\author{Ellen Kirkman}
\address{Wake Forest University, Department of Mathematics and Statistics, P. O. Box 7388, Winston-Salem, North Carolina 27109} 
\email{kirkman@wfu.edu}

\author{W. Frank Moore}
\address{Wake Forest University, Department of Mathematics and Statistics, P. O. Box 7388, Winston-Salem, North Carolina 27109}
\email{moorewf@wfu.edu}

\author{Kewen Peng}
\address{North Carolina State University, Department of Computer Science, Campus Box 8206,
890 Oval Drive,
Engineering Building II,
Raleigh, NC 27695}
\email{kpeng@ncsu.edu}



\begin{abstract}  We present two noncommutative algebras over a field of characteristic zero that each possess a family of actions by cyclic groups of order $2n$, represented in $2 \times 2$ matrices, requiring generators of degree $3n$.
\end{abstract}
\maketitle
 
\section{Background}
Emmy Noether proved the following theorem that is useful in computing the invariants of a finite group  acting linearly on a commutative polynomial ring over a field of characteristic zero (or when the characteristic of $\kk$ is larger than $|G|$).

\begin{theorem}[Noether 1916 \cite{N}] If $\kk$ is a field of characteristic zero and $G$ is a finite group of invertible $n \times n$ matrices acting linearly on $A:= \kk[x_1, \dots, x_n]$ then the ring of invariants $A^G$ can be generated by polynomials of total degree $\leq |G|$. \end{theorem}
The non-modular case (where characteristic of $\kk$ does not divide $|G|$) was proven independently by Fleischmann (2000) \cite{Fl}, Fogarty (2001) \cite{Fo}, and Derksen and Sidman (2004) \cite{DS}.
 
Noether's bound can be sharp in the case $G$ is a cyclic group, and Domokos and Heged\"{u}s  provided a smaller upper bound on the degrees of generators  if $G$ is not cyclic; this result was extended to all characteristics by Sezer.
\begin{theorem}[Domokos and Heged\"{u}s 2000 \cite{DH}, Sezer 2002 \cite{Se}]
 If $G$ is a non-cyclic finite group of invertible $n \times n$ matrices acting linearly on $A:= \kk[x_1, \dots, x_n]$ then the ring of invariants $A^G$ can be generated by polynomials of total degree $\leq 3|G|/4$ if $|G|$ is even, and $\leq 5|G|/8$ if $|G|$ is odd.
\end{theorem}

The Noether bound does not always hold if the field has characteristic $p$ (see, for example, Example 3.5.5(a) p. 94 \cite{DK}, where a degree $3$ invariant is required to generate the  invariants under a group of order $2$ acting on polynomials in $6$ variables over a field of characteristic $2$). For further background on the problem of finding degree bounds for groups acting on $A = \kk[x_1, \dots, x_n]$ see the survey of  Neusel  from  2007 \cite{Neusel}.
Symonds proved the following general theorem that is true for a field in any characteristic, showing that there is an upper bound that is a function of both the order of the group and the dimension of the representation of the group (i.e. the number of variables in the polynomial ring).
\begin{theorem}[Symonds 2011 \cite{S}]
If $G$ is a finite group of order $|G| > 1$ acting linearly on $A:= \kk[x_1, \dots, x_n]$ with $n \geq 2$ then the ring of invariants $A^G$ can be generated by polynomials of degree $\leq n (|G|-1)$.
\end{theorem}
The case of group actions on the noncommutative skew polynomial ring $A=\kk_{-1}[x_1, \dots, x_n]$ (polynomials where $x_jx_i = -x_ix_j$ for $i \neq j$), was considered by Kirkman, Kuzmanovich, and Zhang (\cite{KKZ}), where results concerning permutation actions on $A=\kk_{-1}[x_1, \dots, x_n]$ were proven, and it was noted that the Noether bound does not hold when $n=2$ and $G$ is the group of order 2 generated by the transposition of variables, 
since a generating set of the fixed ring requires a generator of degree $3$. Here the difference between the order of the group and the largest degree needed in a set of generators is only 1, but in this paper we show that this difference can be arbitrarily large. 

Throughout this paper  let  $\kk$ be an algebraically closed field of characteristic 0,  $\i^2 = -1$, and
$G$ be the cyclic group of order $2n$ generated by the matrix 
$g:= \begin{bmatrix} 0 & \lambda\\1 & 0 \end{bmatrix}$
where $\lambda = e^{2 \pi \i/n}$ is a primitive $n$th root of unity,  For an algebra $A$ on which $g$ acts, let $\beta(g)$ denote the minimal degree $d$ such that  the ring of invariants $A^G$ has a set of algebra generators of degree $\leq d$.
We compute $\beta(g)$ explicitly in Section 2 for
 the skew polynomial ring $\kk_{-1}[u,v]$,  and in Section 3 for the down-up algebra $A(0,1)$.  In both cases, when $n$ is odd $\beta(g) = 3n$
(Theorems \ref{n1mod2} and \ref{downupnodd}), the Noether bound does not hold, and $\beta(g)$ is arbitrarily larger than the order of the group. Moreover,  in the case where $A = \kk[x_1, \dots, x_n]$ Noether's bound is exceeded in characteristic $p$ by using a
representation of $G$ of large dimension. For the noncommutative algebras that we consider here the large values of $\beta(g)$ are achieved using only a two-dimensional representation.  Both of these algebras were considered because they are Artin-Schelter regular algebras (as in \cite{AS}), and can be regarded as natural noncommutative generalizations of commutative polynomial rings. On the other hand, F. Gandini \cite[Theorem VI.13]{G} has shown that the Noether bound holds for exterior algebras $\kk_{-1}[x_1, \dots, x_n]/(x_1^2, \dots, x_n^2)$.  The analog of Noether's bound for noncommutative algebras is not yet evident, but these examples provide data for further research.

Many of the paper's computations were provided by the fourth author in an undergraduate research project in mathematics at Wake Forest University.  Computations were aided by the NCAlgebra package in Macaulay2 \cite{M2}.
\section{A cyclic group acting on a skew polynomial ring}

Let $A := \kk_{-1}[u,v]$ be the skew polynomial ring with $vu = -uv$.
The group $G$ acts on $A$ by $g.u = v$ 
and $g.v = \lambda u $.
Let $\mathcal{O}(i,j):=\sum_k g^k.(u^iv^j)$ denote the orbit sum of the monomial $u^iv^j$, i.e. the sum of the distinct elements in the orbit of $u^iv^j$ under the action of $G$. The orbit sum of a monomial is equal, up to a constant, to 
the value of the Reynolds operator
$R_G(u^iv^j) = \sum_{g \in G} g.(u^iv^j)/|G| $, where the sum is taken over all elements of the group.
\begin{proposition}\label{invar}
An element $a\in A^G$ is a linear combination of elements of the form
$\mathcal{O}({kn-i},i) :=u^{kn-i}v^i+(-1)^{(kn-i)i}\lambda^iu^iv^{kn-i},$
for some $k,i$, where $k \geq 1$ and $0 \leq i \leq \lfloor kn/2 \rfloor$.
\end{proposition}
\begin{proof}
Since the elements of $G = (g)$ do not change the total degree of an element of $A$, an invariant is a linear combination of homogeneous elements of some degree $p$; these elements are linear combinations of monomials of the form $u^{p-i}v^i$ for $i = 0, \dots, p$ for some $p$. 
Since $ g^2:= \begin{bmatrix}
   \lambda & 0\\
   0 & \lambda
\end{bmatrix}$, any homogeneous invariant must be a linear combination of invariants of degree a multiple of $n$, indeed $$g^2.(u^{p-i}v^i) = \lambda^p
u^{p-i}v^i = u^{p-i}v^i \Leftrightarrow p\equiv0\mod n.$$
Hence homogeneous invariants are linear combinations of orbit sums of monomials  of the form  $u^{kn-i}v^i$ for some $k,i$, and the orbit of $u^{kn-i}v^i$ is $$\{
u^{kn-i}v^i, (-1)^{(kn-i)i}\lambda^iu^iv^{kn-i}
\}$$ because $g^2.u^{kn-i}v^i = u^{kn-i}v^i$.
\end{proof}
Note that when the monomial is an invariant, the two summands in the proposition above are identical, and we have defined $\mathcal{O}(i,j)$ to be twice the sum of monomials in the orbit (e.g. when $n=2$, we define $\mathcal{O}(2,2) = 2u^2d^2$).  Further, it is possible for $\mathcal{O}(i,j) = 0$ (e.g. when $n=2$, we have $\mathcal{O}(1,1) = uv + vu = 0$).

In this section we will show that
\begin{align*}
\beta(g) =
\begin{cases} n &\text{ if } n \equiv 2 \mod{4}\\
              2n &\text{ if } n \equiv 0 \mod{4}\\
              3n &\text{ if } n \equiv 1 \mod{2}.
            \end{cases}
              \end{align*}
Hence the Noether bound does not hold when $n$ is odd.
\subsection{Case $n$ even}  We begin with the case $n \equiv 2 \mod{4}$.
       \begin{theorem}\label{n2mod4}
If $n \equiv2 \mod 4$ then $\beta(g)=n \lneq |G|=2n$.
\end{theorem}
\begin{proof}
We show that all invariants can be obtained from the invariants of degree $n$, namely from $\mathcal{O}(n-i,i) = u^{n-i}v^i+ (-1)^i\lambda^iu^iv^{n-i},$ for $i=0,\dots,n/2.$
Inducting on $k$ we show these elements generate all invariants of degree $kn$ where $k \geq2$. Note that $\mathcal{O}({n/2},{n/2})/2 = u^{n/2}v^{n/2}$ is an invariant of degree $n$, since $n/2$ is odd and $g.u^{n/2}v^{n/2} = v^{n/2} \lambda^{n/2} u^{n/2} = - \lambda^{n/2} u^{n/2}v^{n/2} = u^{n/2}v^{n/2}$. It suffices to show that the invariants of the form $a=u^{kn-i}v^i+(-1)^i\lambda^iu^iv^{kn-i}$ can be generated. If $i\geq n/2,$ 
\begin{align*}
a&= u^{n/2}v^{n/2}((-1)^{i+1}u^{(2k-1)n/2-i}v^{i-n/2}+(-1)^i(-1)^{i+1}\lambda^iu^{i-n/2}v^{(2k-1)n/2-i})\\
&=(-1)^{i+1}u^{n/2}v^{n/2}(u^{(2k-1)n/2-i}v^{i-n/2}+(-1)^i\lambda^iu^{i-n/2}v^{(2k-1)n/2-i})\\
&=(-1)^{i+1}u^{n/2}v^{n/2}f,
\end{align*} 
where $f$ is the orbit sum of $u^{(2k-1)n/2-i}v^{i-n/2}$, an invariant of degree $(k-1)n$ and hence generated by induction.
If $i < n/2$, $$(u^n+v^n)(u^{(k-1)n-i}v^i+(-1)^i\lambda^iu^iv^{(k-1)n-i})=a+g$$ 
where $g$ is $u^{(k-1)n-i}v^{n+i}+(-1)^i\lambda^iu^{n+i}v^{(k-1)n-i}$ where $n+i \geq n/2$, and $f$ can be generated by the previous case. Hence  we have shown  $\beta(g)=n.$
\end{proof}

   Next we consider the case where $n \equiv 0 \mod{4}$.    
\begin{theorem}\label{n0mod4}
If $n \equiv 0 \mod 4$, then $\beta(g)=2n$.
\end{theorem}
\begin{proof}
We first show that $\beta(g) \leq 2n$.
Let $a\in A^G$ of degree $kn$ for $k \geq2$; we show, by induction on $k$ that $a$ can be generated by elements
of degree $\leq 2n$. 
Note that $u^nv^n \in A^G$ is an invariant of degree $2n$.
Without loss of generality by Proposition \ref{invar} we can assume that $a$ is of
the form $a:= u^{kn-i}v^i + (-1)^i\lambda^i u^iv^{kn-i}$. If $i \geq n$ then we can factor the invariant $u^nv^n$ from $a$, writing $a$ as  $a= (u^nv^n)b$, where $b$ is an invariant of degree $(k-2)n$, and $b$ can
be generated by invariants of degree $\leq 2n$ by induction.  Therefore it suffices to prove that if $i < n$ then $a$ can be generated by
invariants of degree $\leq 2n$.

If $i < n$ note that $\mathcal{O}((k-1)n-i,i) \mathcal{O}(n,0) = \mathcal{O}(kn-i,i) + \mathcal{O}((k-1)n-i,n+i),$
where $\mathcal{O}((k-1)n-i,i)$ can
be generated by invariants of degree $\leq 2n$
by induction and $\mathcal{O}((k-1)n-i,n+i)$
can
be generated by invariants of degree $\leq 2n$
by the case above.
 So $\beta(g) \leq 2n$.

Finally, we show $\beta(g) \geq 2n$, i.e. $A^G$ cannot be generated by linear combinations of elements of degree $n$. Invariants of degree $n$ are linear combinations of elements of the form $\mathcal{O}(u^{n-i}v^i):=u^{n-i}v^i+(-1)^i\lambda^iu^iv^{n-i}$
for $i=0,\dots,n/2-1$.
Note that $u^{n/2}v^{n/2}$ is not invariant because $g.u^{n/2}v^{n/2} = \lambda^{n/2}v^{n/2}u^{n/2}
=-v^{n/2}u^{n/2} =-u^{n/2}v^{n/2}$ because $n/2$ is even. These $n/2$ invariants are linearly independent.

One checks that when $n$ is even these orbit sums satisfy the equation:
\begin{align}
\label{4n}
\mathcal{O}(n-i,i) \mathcal{O}(n-j,j)& = (-1)^{ij} \mathcal{O}(2n-i-j,i+j) \nonumber\\ &+
\begin{cases} (-1)^{i(j+1)}\lambda^i\mathcal{O}(n +i-j,n-i+j) & 0\leq j \leq i < n/2\\
 (-1)^{j(i+1)} \lambda^j\mathcal{O}(n +j-i,n-j+i) & 0\leq i \leq j < n/2,\\
 \end{cases}\end{align}
 so $\mathcal{O}(n-i,i) \mathcal{O}(n-j,j) =\mathcal{O}(n-j,j) \mathcal{O}(n-i,i).$
 Hence in considering the products of invariants of degree $n$ we may assume that $i \geq j$.
 We will show that one cannot write all orbit sums of degree  $2n$ as linear combinations 
of products  of orbit sums of degree $n$ by showing that the homogeneous system of $\binom{n/2+1}{2}$ equations in $n$ variables (\ref{4NSys}) below has a nontrivial solution.  The coefficient matrix of this system is the coefficients of the orbit sums of degree $2n$ in equation (\ref{4n}) above. The existence of this solution shows that
the dimension of the space spanned by products of
two invariants of degree $n$ is $<$ the dimension of the space spanned by the orbit sums of degree $2n$. Hence we next show that the system
\begin{equation}\label{4NSys}
0  = (-1)^{ij}x_{i+j} + (-1)^{i(j+1)}\lambda^ix_{n-i+j}\quad  0\leq j \leq i < n/2\\
 \end{equation}
 in the variables $x_p$ for $p=0,\ldots, n-1$ has a nontrivial solution given by
\begin{equation}\label{4NSysSol}
x_p=(-1)^{\lfloor\frac{p+1}{2}\rfloor}\lambda^{\lfloor\frac{p+1}{2}\rfloor}.
\end{equation}



 First we consider the case $i = 2\ell$ and $j=2k-1$ for $\ell = 0, \dots, n/2 -1$ and $k= 1, \dots, n/2$, and without loss of generality we can assume $i \geq j$. In this case the equation (\ref{4NSys})  becomes
$0=x_{2(l+k)-1}+\lambda^{2l}x_{2(n/2-l+k)-1}.$
Substituting the expression in \eqref{4NSysSol} for the variables we get 
\begin{align*}&(-1)^{\ell + k}\lambda^{\ell + k}+ \lambda^{2\ell} (-1)^{(n/2- \ell +k)} \lambda^{(n/2-\ell + k)}\\&=
 (-1)^{\ell + k}\lambda^{\ell + k}+ \lambda^{2\ell} (-1)^{(- \ell +k)} \lambda^{n/2}\lambda^{(-\ell + k)}\\
 &= (-1)^{\ell + k}\lambda^{\ell + k}+  (-1)^{(- \ell +k)} (-1)\lambda^{(\ell + k)}\\
 &=  (-1)^{\ell + k}\lambda^{\ell + k}+  (-1)^{( \ell +k)} (-1)\lambda^{(\ell + k)} = 0. 
\end{align*}

The cases $i,j$ both even or both odd are handled similarly. 
   Hence the products of degree $n$ orbit sums cannot determine all the orbit sums $\mathcal{O}(2n-i, i)$, and hence invariants of degree $n$ do not generate all invariants of degree $2n$.
\end{proof}

\subsection{Case $n$ odd} When $n$ is odd we prove that the Noether bound does not hold.

\begin{lemma}\label{multi} If $n$ is odd then an orbit sum of degree $2n$ and an orbit sum of degree $n$ multiply as 
\begin{align*}
\mathcal{O}(2n-i,i)\mathcal{O}(n-j,j)&=(-1)^{i(j-1)}\mathcal{O}(3n-i-j,i+j)\\
&+\begin{cases}(-1)^{ij}\lambda^j\mathcal{O}(2n-i+j,n+i-j)&i-j<\frac{n}{2}\\(-1)^{ij}\lambda^i\mathcal{O}(n+i-j,2n+j-i)&i-j>\frac{n}{2}\end{cases},
\end{align*}
\begin{align*}
\mathcal{O}(n-j,j)\mathcal{O}(2n-i,i)&=(-1)^{ij}\mathcal{O}(3n-i-j,i+j)\\
&+\begin{cases}(-1)^{i(j-1)}\lambda^j\mathcal{O}(2n-i+j,n+i-j)&i-j<\frac{n}{2}\\(-1)^{i(j-1)}\lambda^i\mathcal{O}(n+i-j,2n+j-i)&i-j>\frac{n}{2}\end{cases}.
\end{align*}
\end{lemma}
\begin{proof}
We calculate $\mathcal{O}(2n-i,i)\mathcal{O}(n-j,j)$ when $i-j<\frac{n}{2}$, the other cases are similar,
\begin{align*}
\mathcal{O}(2n-i,i)\mathcal{O}(n-j,j)&=(u^{2n-i}v^i+(-1)^i\lambda^i u^iv^{2n-i})(u^{n-j}v^j+\lambda^j u^jv^{n-j})\\
&=(-1)^{i(j-1)}u^{3n-i-j}v^{i+j}+(-1)^{i(j-1)}\lambda^{i+j}u^{i+j}v^{3n-i-j}\\
&\hphantom{=}+(-1)^{ij}\lambda^ju^{2n-i+j}+(-1)^{ij}\lambda^iu^{n-j+i}v^{2n-i+j}.
\end{align*}
Now it suffices to notice that
\[
(-1)^{i(j-1)}u^{3n-i-j}v^{i+j}+(-1)^{i(j-1)}\lambda^{i+j}u^{i+j}v^{3n-i-j}=(-1)^{i(j-1)}\mathcal{O}(3n-i-j,i+j),
\]
and if $i-j<\frac{n}{2}$ then $2n-i+j>n+i-j$ and therefore
\[
(-1)^{ij}\lambda^ju^{2n-i+j}+(-1)^{ij}\lambda^iu^{n-j+i}v^{2n-i+j}=(-1)^{ij}\lambda^j\mathcal{O}(2n-i+j,n+i-j).
\]
\end{proof}
Notice that when $i$ is even the two expressions in Lemma \ref{multi} coincide; this is expected since in this case $\mathcal{O}(2n-i,i)$ is central. When $i$ is odd then the two products in Lemma \ref{multi} differ by a sign.

\begin{theorem}\label{n1mod2}
If $n \equiv1 \mod 2$ then $\beta(g)=3n \gneq |G|=2n$.
\begin{proof}

We first show that $\beta(g) \leq 3n$.
It suffices to show by induction on $k$ that for $k \geq 4$, an invariant of the form $a = \mathcal{O}({kn-i}, i)$ can be generated by invariants of degree $\leq 3n.$
Note that $u^nv^n$ is not invariant since $g.u^n v^n = \lambda^nv^nu^n
=v^nu^n =-u^n v^n$ because $n$ is odd.
 
First we show that we can generate the invariant $u^{2n}v^{2n}.$  We have
the equations:
$$(u^n+v^n)(u^{3n}+v^{3n})=(u^{4n}+v^{4n})-(u^{3n}v^n-u^nv^{3n})$$
$$(u^{3n}+v^{3n})(u^n+v^n)=(u^{4n}+v^{4n})+(u^{3n}v^n-u^nv^{3n}).$$
Adding the two equations above shows that $u^{4n}+v^{4n}$ can be generated by invariants of degree $\leq 3n.$ Using the equation $(u^{2n}+v^{2n})^2=u^{4n}+v^{4n}+2u^{2n}v^{2n}$ we see that
$u^{2n}v^{2n}$ can be generated by invariants of degree $\leq3n$.

Next we show that we can generate any orbit sum $a=\mathcal{O}({kn-i},i)$ for $k\geq4$ if $i \geq2n$.   In this case we can factor out $u^{2n}v^{2n}$ and obtain
\begin{align*}
a&=u^{2n}v^{2n}(u^{(k-2)n-i}v^{i-2n}+(-1)^{i(k+1)}\lambda^{i}u^{i-2n}v^{(k-2)n-i})\\
&=u^{2n}v^{2n}\mathcal{O}({(k-2)n-i},i-2n),
\end{align*}
and since $\mathcal{O}({(k-2)n-i},i-2n)$ has total degree $(k-4)n$ it can be generated by invariants of degree $\leq 3n$ by induction, and therefore so can $a$.
If $i<2n$ we have the equation $$(u^{2n}+v^{2n})\mathcal{O}({(k-2)n-i},i)=\mathcal{O}({kn-i},i)+\mathcal{O}({(k-2)n-i},{2n+i}),$$
and $\mathcal{O}({(k-2)n-i},i)$ is of degree $(k-2)n$, so generated by invariants of degree $\leq 3n$ by induction. 
Further $\mathcal{O}({(k-2)n-i},{2n+i})$ is generated by invariants of degree $\leq 3n$ by the first case. Therefore $a=\mathcal{O}({kn-i},i)$ is generated by invariants of degree $\leq 3n$, and we have shown $\beta(g)\leq3n.$

Next we show $\beta(g) \geq 3n$.
We claim that the vector with entries
\[
x_k=(-1)^{\frac{(\frac{3n-1}{2}-k)(\frac{3n-1}{2}-k+1)}{2}}(\lambda^{\frac{n-1}{2}})^{\frac{3n-1}{2}-k}
\]
 is a solution to the equations 
\[
0=(-1)^{ij}x_{i+j}+\begin{cases}(-1)^{i(j-1)}\lambda^jx_{n+i-j}&i-j<\frac{n}{2}\\(-1)^{i(j-1)}\lambda^ix_{2n+j-i}&i-j>\frac{n}{2}\end{cases}.
\]
We start by checking the first equation in the system  in Lemma $\ref{multi}$
\begin{align*}
0&=(-1)^{ij}(-1)^{\frac{(\frac{3n-1}{2}-i-j)(\frac{3n-1}{2}-i-j+1)}{2}}(\lambda^{\frac{n-1}{2}})^{\frac{3n-1}{2}-i-j}\\
&+(-1)^{i(j-1)}\lambda^j(-1)^{\frac{(\frac{3n-1}{2}-n-i+j)(\frac{3n-1}{2}-n-i+j+1)}{2}}(\lambda^{\frac{n-1}{2}})^{\frac{3n-1}{2}-n-i+j}.
\end{align*}
We first check that the signs of the two summands are opposite, we need the following congruence to be true
\begin{align*}
&ij+\frac{(\frac{3n-1}{2}-i-j)(\frac{3n-1}{2}-i-j+1)}{2}\equiv\\ &i(j-1)+\frac{(\frac{3n-1}{2}-n-i+j)(\frac{3n-1}{2}-n-i+j+1)}{2}+1\;(\mathrm{mod}\;2),
\end{align*}
this congruence is equivalent to
\begin{align*}
&\frac{(\frac{3n-1}{2}-i-j)(\frac{3n-1}{2}-i-j+1)-(\frac{3n-1}{2}-n-i+j)(\frac{3n-1}{2}-n-i+j+1)}{2}\equiv\\ &-i+1\;(\mathrm{mod}\;2),
\end{align*}
simplifying the left side yields
$
2ji-2jn-in+n^2\equiv -i+1\;(\mathrm{mod}\;2),
$
which is true since $n$ is odd.

Now we check that the power of $\lambda$ is the same; for that to be true we need
\[
(\frac{n-1}{2})(\frac{3n-1}{2}-i-j)\equiv j+(\frac{n-1}{2})(\frac{3n-1}{2}-n-i+j)\;(\mathrm{mod}\;n),
\]
since $n$ is odd $2$ is invertible, hence multiplying by $4$ we get
\[
(n-1)(3n-1-2i-2j)\equiv 4j+(n-1)(3n-1-2n-2i+2j)\;(\mathrm{mod}\;n),
\]
which is equivalent to
$1+2i+2j\equiv 4j+1+2i-2j\;(\mathrm{mod}\;n)$,
which is true.

The second equation in Lemma $\ref{multi}$
can be checked with a similar computation.
The above computations show that one cannot solve for all generators of degree $3n$ using lower degree invariants, as the system of equations needed to solve for these orbit sum invariants of degree $3n$ has rank less than the number of variables, since we have shown that kernel of the coefficient matrix  has a nontrivial element.
 Hence invariants of degree $\leq 2n$ do not generate all invariants of degree $3n$, and  $\beta(g)=3n \gneq |G|=2n$. 
\end{proof}
\end{theorem}
\section{A cyclic group acting on a graded down-up algebra}
The down-up algebras were introduced by Benkart and Roby in \cite{BR}, as a generalization of the universal enveloping algebra of some three-dimensional Lie algebras.  Let $\alpha$ and $\beta$ be fixed elements of a field $\kk$.  The graded down-up algebras $A(\alpha, \beta)$ are the algebras generated over $\kk$ generated by two elements $u, d$ with relations
$$d^2u = \alpha dud + \beta ud^2 \quad \text{ and } \quad 
du^2 = \alpha udu + \beta u^2d.$$
These algebras have a monomial basis of the form
$u^a (du)^b d^c$, and we define the degree of
this monomial to be the total degree in $u$ and $d$, i.e. $a + 2b +c$.
The down-up algebra $A(\alpha, \beta)$ is  noetherian if and only if it is Artin-Schelter regular if and only if $\beta \neq 0$ \cite{KMP}. Again let $ \lambda = e^{2\pi\i/n}$ be a primitive $n$th root of unity. Denote
by $g$ the automorphism of $A(\alpha, \beta)$
with $g.u = d$ 
and $g.d = \lambda u $, and let
 $G:=(g)$ be the cyclic group of order $2n$ generated by $g$. The graded automorphism groups of $A(\alpha, \beta)$ were computed in \cite[Proposition 1.1]{KK}, and the graded down-up algebras on which $g$ acts are  $A(0,1), A(0,-1)$ and $A(2,-1)$.
 As in the previous section, any invariant under $G$ must have degree a multiple of $n$.
Let $\mathcal{O}(u^a(du)^b d^c):=\sum_k g^k.(u^a(du)^b d^c)$ denote the orbit sum of the monomial $u^a(du)^b d^c$.
Orbits of monomials will consist of one or two summands;  when the monomial is invariant
$\mathcal{O}(u^a(du)^b d^c)$ will denote twice the orbit sum. 
Further, we will sometimes denote the orbit sum $\mathcal{O}(u^a(du)^b d^c)$ as $\mathcal{O}(a,b,c)$.  

In this section we consider the case $\alpha =0, \beta=1 $, so that the relations in $A = A(0,1)$ are $u^2 d = du^2$ and $d^2u = ud^2$.
We show that
\begin{align*}
\beta(g) = \begin{cases} 2n &\text{ if } n \equiv 0 \mod{2}\\
                         3n & \text{ if } n \equiv 1 \mod{2},
                         \end{cases}
                         \end{align*}
                         and hence in the case that $n$ is odd the Noether bound does not hold.
\subsection{Case $n$ even} We begin by proving an upper bound on $\beta(g)$ when $n$ is even.

\begin{proposition}
If $n \equiv0 \mod 2$ then $\beta(g) \leq |G|=2n$.
\end{proposition}
\begin{proof}
First, we show that $\beta(g) \leq2n$. We know that $u^n d^n$ is an invariant of degree $2n$, because $g.u^n d^n=d^n u^n=u^n d^n$. 
We prove by induction on $k$ that invariants 
of degree $ kn$ for $k \geq 3$
can be degenerated by invariants of
degree $\leq 2n$.
Let $\mathcal{O}(u^a(du)^bd^c)$ be an orbit
sum of degree $a + 2b +c = kn$ for $k \geq 3$.\\
\underline{\bf Case 1}: $a\geq n$ and $c\geq n$.  Then
 $\mathcal{O}(u^{a-n}(du)^b d^{c-n})\mathcal{O}(u^n d^n) = 
\mathcal{O}(u^a(du)^b d^c)$, so that
$\mathcal{O}(u^a(du)^b d^c)$ can be generated by two invariants of lower degree.\\
\underline{\bf Case 2}: $a\geq 2n$ (or by symmetry $c\geq 2n$). Then
 $\mathcal{O}(u^{a-n}(du)^b d^{c})\mathcal{O}(u^n) = 
\mathcal{O}(u^a(du)^b d^c)+\mathcal{O}(u^{a-n}(du)^b d^{c+n})$ and 
the second invariant on the right-hand side of the equation belongs to Case 1 because $(a-n) \geq n$ and $(c+n) \geq n$.  Hence
$\mathcal{O}(u^a(du)^b d^c)$ can be generated by two invariants of lower degree. \\
\underline{\bf Case 3}: $b=0$. If $0<c<n$ then
 \[
\mathcal{O}(u^{(k-2)n} )\mathcal{O}(u^{2n-c} d^c) = \mathcal{O}(u^{kn-c} d^c) + \mathcal{O}(u^{2n-c}d^{(k-2)n + c}),
\]
and the second invariant on the right-hand side belongs to Case 1. Hence
$\mathcal{O}(u^{kn-c} d^c)$ can be generated by two invariants of lower degree, and by symmetry this is true if $0 < a < n$.

If $b=0$ and $c \geq n$, then either $a \geq n$ and we are done by Case 1 or $a <n$ and we are done by symmetry as noted above.\\
\underline{\bf Case 4}: When $c=0$ (and by symmetry when $a=0$).\\
\underline{\bf Case 4.1}: When $c=0$ and $n \leq b$. Then
$\mathcal{O}(u^{kn-2b} (du)^{b-n/2} )\mathcal{O}((du)^{n/2}) = \mathcal{O}(u^{kn-2b} (du)^b)+ \lambda^{n/2}\mathcal{O}(u^{(k+1)n-2b}(du)^{b-n} d^n )$, where the second invariant on the right-hand side belongs to Case 1.  Hence $\mathcal{O}(u^{kn-2b} (du)^b)$ can be generated by two invariants of lower degree.\\
\underline{\bf Case 4.2}: $c=0$ and $n/2 \leq b < n \leq kn/2$. Then
$$\mathcal{O}(u^{kn-2b}(du)^{b-n/2})\mathcal{O}((du)^{n/2})= \mathcal{O}(u^{kn-2b}(du)^b) + \lambda^{n/2}\mathcal{O}(u^{(k-1)n + 1}(du)^{n-b-1} d^{2b-n+1})$$ where the second invariant on the  belongs to Case 2, since $(kn-n)+1 \geq 2n$.
Then $\mathcal{O}(u^{kn-2b} (du)^b)$ can be generated by two invariants of lower degree:\\
\underline{\bf Case 4.3}: $c=0$ and  $b < n/2$.  Then 
if $b$ is even:
$$\mathcal{O}(u^{(k-1)n-b}(du)^{b/2} )\mathcal{O}(u^{n-b} (du)^{b/2}) = \mathcal{O}(u^{kn-2b} (du)^b) + \lambda^{b/2}\mathcal{O}(u^{(k-1)n} d^n)$$
and if $b$ is odd:
$$\mathcal{O}(u^{(k-1)n-b+1}(du)^{\frac{b-1}{2}})\mathcal{O}(u^{n-b-1} (du)^\frac{b+1}{2}) = \mathcal{O}(u^{kn-2b} (du)^b) + \lambda^{\frac{b+1}{2}}\mathcal{O}(u^{kn-n+1} d^{n-1}).$$
In both cases, the second invariant on the right-hand side belongs to Case 2 and
$\mathcal{O}(u^{kn-2b} (du)^b)$ can be generated by two invariants of lower degree.\\
\underline{\bf Case 5}:  $a,b,c \neq 0$.\\
\underline{\bf Case 5.1}: $b \geq n$. Then :\[
\begin{cases}
\mathcal{O}(u^a (du)^{b- \frac{n}{2}} d^c )\mathcal{O}((du)^{n/2}) = \mathcal{O}(u^a (du)^b d^c) + \lambda^{\frac{n}{2}}\mathcal{O}((u^{a+n} (du)^{b-n} d^{c+n})\quad\mathrm{if}\;c\;\mathrm{even}\\
\mathcal{O}(u^a (du)^{b- \frac{n}{2}} d^c)\mathcal{O}(u (du)^{\frac{n}{2}-1} d) = \mathcal{O}(u^a (du)^b d^c) + \lambda^{\frac{n}{2}}\mathcal{O}(u^{a+n}(du)^{b-n}d^{c+n})\quad\mathrm{if}\;c\;\mathrm{odd}.\\
\end{cases}
\]
In both cases, the second invariant on the right-hand side belongs to Case 1, and $\mathcal{O}(u^a (du)^b d^c)$ can be generated by two invariants of lower degree.\\
\underline{\bf Case 5.2}:  $b < n$. We know that with both $a \geq n$ and $c \geq n$, the case falls into case (1). Hence we assume $a < n$, and hence $(2b+c)>2n$, so $(b+c)>n$, and there are 4 subcases as shown below.
If $b$ and $c$ are both even:
$$\mathcal{O}(u^a (du)^\frac{b}{2} d^{b+c-n} )\mathcal{O}((du)^{b/2} d^{n-b}) = \mathcal{O}(u^a (du)^b d^c) + \lambda^{-\frac{b}{2}}\mathcal{O}(u^{a+n} d^{2b+c-n} ).$$
If $b$ is odd and $c$ is even:
$$\mathcal{O}(u^a (du)^\frac{b-1}{2} d^{c+b+1-n})\mathcal{O}((du)^\frac{b+1}{2} d^{n-b-1}) = \mathcal{O}(u^a (du)^b d^c)+\lambda^{-\frac{b+1}{2}} \mathcal{O}(u^{a+n-1} d^{2b+c-n+1}).$$
If $b$ is even and $c$ is odd:
$$\mathcal{O}(u^a (du)^\frac{b}{2} d^{c+b-n})\mathcal{O}(u(du)^{\frac{b}{2}-1} d^{n-b+1}) = \mathcal{O}(u^a (du)^b d^c)+ \lambda^{-\frac{b}{2}} \mathcal{O}(u^{a+n} d^{2b+c-n}).$$
If $b$ and $c$ are both odd (note that $c + 2b-n+1<2n-n+1$ and $a \geq 1$):
$$\mathcal{O}(u^a(du)^{\frac{b-1}{2}}d^{b+c-n+1}) \mathcal{O}(u(du)^{\frac{b-1}{2}}d^{n-b}) = \mathcal{O}(u^a(du)^bd^c) + \lambda^{-\frac{b-1}{2}}\mathcal{O}(u^{a+n-1}d^{c+2b -n+1})$$
In all four cases we are done by Case 3, and by symmetry we are done if $c< n$, completing the proof of Case 5.2.
\end{proof}
From now on we will denote the element $\mO(u^{kn-2i-j}(du)^id^j)$ as $\mathcal{O}(kn-2i-j,i,j)$. We fix a basis for the vector space generated by the orbit sums of degree $2n$. The orbit sum $\mathcal{O}(2n-2i-j,i,j)$ is a basis element if $n>i+j$ or if $n=i+j$ and $i$ is even. If an orbit sum is represented by a triple not satisfying the conditions above, we can switch it to one that does using the following formula
\begin{equation} \label{nEvenBasis}
\mathcal{O}(2n-2i-j,i,j)=\lambda^{i+j}\begin{cases}\mathcal{O}(j-1,i+1,2n-2i-j-1)\quad j\;\mathrm{odd}\\ \mathcal{O}(j+1,i-1,2n-2i-j+1)\quad j\;\mathrm{even}\;i\neq0\\ \mathcal{O}(j,0,2n-2i-j)\quad\hphantom{-2i-j+1} j\mathrm{\;even}\;i=0.\end{cases}
\end{equation}

\begin{proposition} \label{nEvenOnTimesOn}
The product of the orbit sums $\mathcal{O}(n-2i-j,i,j)\mathcal{O}(n-2p-q,p,q)$ is: if $j$ and $q$ are even then the product is
\begin{align*}
&\mathcal{O}(2n-2i-j-2p-q,p+i,q+j)+\\&\lambda^{p+q}\begin{cases}\mathcal{O}(n-j+q+1,p-i-1,n+2i+j-2p-q+1)\quad p>i\\\mathcal{O}(n-2i-j+q+2p,i-p,n-q+j)\quad\hphantom{2p-q+1}\;\;\; p\leq i.\end{cases}
\end{align*}
If $j$ is even and $q$ is odd then the product is 
\begin{align*}
&\lambda^{p+q}\mathcal{O}(n-2i-j+q-1,i+p+1,n-2p-q+j-1)+\\&\begin{cases}\mathcal{O}(2n-j-2p-q,p-i,2i+j+q)\quad \hphantom{j+q+1)}\;\;\;p\geq i\\\mathcal{O}(2n-2i-j-q+1,i-p-1,2p+j+q+1)\quad p<i.\end{cases}
\end{align*}
If $j$ is odd and $q$ is even then the product is
\begin{align*}
&\lambda^{p+q}\mathcal{O}(n+q-2i-j,i+p,n+j-2p-q)+\\&\begin{cases}\mathcal{O}(2n-j-2p-q+1,p-i-1,q+2i+j+1)\quad p>i\\\mathcal{O}(2n-2i-j-q,i-p,2p+q+j)\quad \hphantom{+2i+j+1}p\leq i.\end{cases}
\end{align*}
If $j$ and $q$ are both odd then the product is
\begin{align*}
&\mathcal{O}(2n-2i-j-2p-q-1,i+p+1,q+j-1)+\\&\lambda^{p+q}\begin{cases}\mathcal{O}(n+q-j,p-i,n+2i+j-2p-q)\quad \hphantom{q+j+1)+}p\geq i\\\mathcal{O}(n-2i-j+2p+q+1,i-p-1,n-q+j+1)\quad p< i.\end{cases}
\end{align*}
\end{proposition}
\begin{proof}
We show the case $j,q$ even and $p\geq i$, the remaining cases are similar. The product
$u^{n-2i-j}(du)^id^ju^{n-2p-q}(du)^pd^q$
is equal to $u^{2n-2i-j-2p-q}(du)^{i+p}d^{j+q}$ since the powers of $u$ and $d$ are central. This, together with the product
\[
\lambda^{i+j}d^{n-2i-j}(ud)^iu^j\cdot\lambda^{p+q}d^{n-2p-q}(ud)^pu^q,
\]
is equal to $\mathcal{O}(2n-2i-j-2p-q,p+i,q+j)$. The product
\begin{equation}\label{OrbitSum2n}
u^{n-2i-j}(du)^id^j\cdot\lambda^{p+q}d^{n-2p-q}(ud)^pu^q
\end{equation}
is equal to $\lambda^{p+q}u^{2n-2i-j+q}(du)^i(ud)^pd^{n-2p-q+j}$. Since $p>i$ we have
\[
(du)^i(ud)^p=u^{2i}d^{2i}(ud)^{p-i},
\]
therefore, since $u^{2i}$ and $d^{2i}$ are central, equation \eqref{OrbitSum2n} reduces to
\[
\lambda^{p+q}u^{n-j+q}(ud)^{p-i}d^{n+2i+j-2p-q}=\lambda^{p+q}u^{n-j+q+1}(du)^{p-i-1}d^{n+2i+j-2p-q+1}.
\]
This, together with the remaining product, form the desired orbit sum.
\end{proof}
Next we prove the lower bound on $\beta(g)$ when $n$ is even.
\begin{theorem}
If $n$ is even then $\beta(g)=2n$.
\end{theorem}
\begin{proof}
It suffices to show that $\beta(g)\geq 2n$, i.e. the orbit sums of degree $2n$ cannot be all generated by orbit sums of degree $n$. In the formulae in Proposition \ref{nEvenOnTimesOn} we replace the orbit sum $\mathcal{O}(2n-2l-k,l,k)$ with the variable $x_{l,k}$, keeping in mind that the triple may need to be changed using \eqref{nEvenBasis} if it is not in the desired form. We also change the product $\mathcal{O}(n-2i-j,i,j)\mathcal{O}(n-2p-q,p,q)$ with zero. This gives rises to a linear system of homogeneous equations. To prove the theorem we prove that this system has a nonzero solution. We claim that the vector $x_{l,k}=\lambda^{-\lfloor\frac{n-l-k}{2}\rfloor}$ is a nonzero solution. We check the case $j,q$ even, $p\geq i$, $p\equiv i\;(\mathrm{mod}\;2)$ and no changes to the triple occurred, the remaining cases are checked similarly. We need to check that
$x_{p+i,q+j}+\lambda^{p+q}x_{p-i-1,n+2i+j-2p-q+1}=0.
$
Under the hypothesis $p\equiv i\;(\mathrm{mod}\;2)$ this yields
$
\lambda^{-\frac{n-p-i-q-j}{2}}+\lambda^{\frac{i+j+p+q}{2}}=0,$
which is true since $\lambda^{\frac{n}{2}}=-1$.
\end{proof}

\subsection{Case $n$ odd}  We next show that when $n$ is odd $\beta(g) = 3n$, and hence the Noether bound does not hold.  

We fix a basis for the orbit sums of degree $3n$ as follows: $\mO(3n-2l-k,l,k)$ is a basis element if $3n>2(l+k)$. If a triple does not have this form, then it can be changed according to the following formula:
\begin{equation}\label{nOddBasis}
\mO(3n-2l-k,l,k)=\lambda^{l+k}\mO(k,l,3n-2l-k).
\end{equation}

\begin{proposition} \label{nOddOnTimesO2n}
The product of the orbit sums $\mathcal{O}(n-2i-j,i,j)\mO(2n-2p-q,p,q)$ is: if $j$ is even and $q$ is odd then the product is
\begin{align*}
&\lambda^{p+q}\mO(n-2i-j+q-1,i+p+1,2n-2p-q+j-1)+\\&\begin{cases}\mathcal{O}(3n-j-2p-q,p-i,q+j+2i)\quad \hphantom{+2p+q+}\;\;\;p\geq i\\\mathcal{O}(3n-2i-j-q+1,i-p-1,j+2p+q+1)\quad p< i.\end{cases}
\end{align*}
If $j$ and $q$ are both odd then the product is
\begin{align*}
&\mO(3n-2i-j-2p-q-1,i+p+1,j+q-1)+\\&\lambda^{p+q}\begin{cases}\mathcal{O}(n-j+q,p-i,2n-2p-q+j+2i)\quad \hphantom{+2p+q+}\;\;\;p\geq i\\\mathcal{O}(n-2i-j+q+2p+1,i-p-1,2n-q+j+1)\quad p< i.\end{cases}
\end{align*}
If $j$ and $q$ are both odd then the product is
\begin{align*}
&\mO(3n-2i-j-2p-q,i+p,j+q)+\\&\lambda^{p+q}\begin{cases}\mathcal{O}(n-2i-j+1+2p,i-p,2n+j-q)\quad \hphantom{+2p+q+}\;\;\;p\leq i\\\mathcal{O}(n-j+q+1,p-i-1,2n+2i+j-2p-q+1)\quad p> i.\end{cases}
\end{align*}
If $j$ is odd and $q$ is even then the product is
\begin{align*}
&\lambda^{p+q}\mO(n-2i-j+q,i+p,2n-2p-q+j)+\\&\begin{cases}\mathcal{O}(3n-2i-j-q,i-p,q+j+2p)\quad \hphantom{+2p+q+}\;\;\;p\leq i\\\mathcal{O}(3n-2p-j-q+1,p-i-1,j+2i+q+1)\quad p> i.\end{cases}
\end{align*}
The product of the orbit sums $\mO(2n-2p-q,p,q)\mO(n-2i-j,i,j)$ is:
\begin{align*}
&\begin{cases}
 \mO(n-2i-j,i,j)\mO(2n-2p-q-1,p+1, q-1) \text{ if }q \text{ odd }\\
 \mO(n-2i-j,i,j) \mO(2n-q,0, q)\hphantom{-1,p+1,q-1)}\; \text{ if } q \text{ even and } p= 0\\
 \mO(n-2i-j, i, j)\mO(2n-2p-q+1, p-1, q+1) \text{ if } q \text{ even and } p \neq 0
 \end{cases}
 \end{align*}
\end{proposition}
The proof of the previous proposition is similar to the proof of Proposition \ref{nEvenOnTimesOn} and therefore is omitted.

\begin{proposition} \label{dulowerboundnodd}
If $n$ is odd then $\beta(g) \geq 3n$.
\end{proposition}
\begin{proof}

We need to show that the orbit sums of degree $3n$ cannot be generated with orbit sums of lower degree. In the formulae in Proposition \ref{nOddOnTimesO2n} we replace the orbit sum $\mO(3n-2l-k,l,k)$ with the variable $x_{l,k}$, keeping in mind that the triple may need to be changed using \eqref{nOddBasis} if not in the desired form. We also change the product $\mathcal{O}(n-2i-j,i,j)\mathcal{O}(2n-2p-q,p,q)$ with zero. This gives rises to a linear system of homogeneous equations. To prove the proposition we prove that this system has a nonzero solution. We claim that the vector $$x_{l,k} = (-1)^{((l+k)(l+k +n) + l(l+1))/2}\lambda^{(n+1)(n+1+2l + 2k)/4}.$$ is a nonzero solution. We check the case $j$ even, $q$ odd, $p\geq i$ and no changes to the triple occurred, the remaining cases are checked similarly. We need to check that
\[
\lambda^{p+q}x_{i+p+1,2n-2p-q+j-1}+x_{p-i,q+j+2i}=0.
\]
We first prove that
\begin{align*}
&\frac{(2n+i+j-p-q)(3n+i+j-p-q)+(i+p+1)(i+p+2)}{2}\equiv\\&\frac{(p+q+i+j)(p+q+i+j+n)+(p-i)(p-i+1)}{2}+1\;(\mathrm{mod}\;2).
\end{align*}
Indeed, computing the left hand side minus the right hand side in $\mathbb{Z}$ yields
\[
j(2n-2p-2q)+3n^2+n(-3p-3q+2i)+p-2iq+2i,
\]
which is zero modulo 2. It remains to prove that
\begin{align*}
&p+q+\frac{(n+1)(n+1+2(i+p+1)+2(2n-2p-q+j-1)}{4}\equiv\\&\frac{(n+1)(n+1+2(p-i)+2(q+j+2i))}{4}\;(\mathrm{mod}\;n).
\end{align*}
Computing the left hand side minus the right hand side in $\mathbb{Z}$ yields $-n(-n+p+q-1)$ which is clearly zero modulo $n$.
\end{proof}

Next we prove that when $n$ is odd, then $\beta(g) \leq 4n$.
\begin{proposition}
\label{4nbound}
If $n \equiv1 \mod 2$ then $\beta(g) \leq 4n$.
\end{proposition}
\begin{proof}
We show that all invariants can be generated by invariants of degrees $\leq 4n$. First of all, we notice that $u^{2n} v^{2n} \in A^G$ is an invariant of degree $4n$, because $g.u^{2n} d^{2n}=d^{2n} u^{2n}=u^{2n} d^{2n}$. We argue by induction, and show that for $a + 2b + c =kn$, with $k\geq 5$, the orbit sum $\mO(a,b,c)$ can be generated by invariants of smaller degree.\\
\noindent
\underline{\bf Case 1}: $a\geq 2n$ and $c\geq 2n$. 
Since $\mathcal{O}(a-2n,b,c-2n)\mathcal{O}({2n},0, {2n}) = 
2\mathcal{O}(a,b, c)$, by induction
$\mathcal{O}(a,b,c)$ can be generated by two invariants of lower degree.\\
\underline{\bf Case 2}: $a\geq 4n$ or $c\geq 4n$.  We may assume $a\geq 4n$, and then $$\mathcal{O}({a-2n}, b,{c})\mathcal{O}({2n},0,0) = 
\mathcal{O}(a,b,c)+ \mathcal{O}({a-2n}, b, {c+2n}).$$ The second invariant on the right-hand side belongs to Case 1 because $(a-2n) \geq 2n$ and $(c+2n) \geq 2n$. \\
\underline{\bf Case 3}: $b\geq 2n$. Then $\mathcal{O}(a,{b-n},{c})\mathcal{O}(0, {n}, 0) = \mathcal{O}(a, b,  c)+ \mathcal{O}({a+2n},{b-2n}, {c+2n}).$ 
The second invariant on the right-hand side belongs to Case 1 because $(a+2n) \geq 2n$ and $(c+2n) \geq 2n$.\\
\underline{\bf Case 4}: $b=0$.  We consider $\mathcal{O}(a, 0, c)$, where $a+c=kn \geq 5n$. Let $a \geq c$. Then we have the following cases.\\ 
\underline{\bf Case 4.1}: If $c \geq 2n$, the case falls into Case 1.\\
\underline{\bf Case 4.2}: If $c \leq n$, the case falls into Case 2 since $a \geq 4n$.\\
\underline{\bf Case 4.3}: If $n < c < 2n$, where $a \geq 3n$, we have:
\[
\mathcal{O}({a-n},0,c)\mathcal{O}(n,0,0) = \mathcal{O}(a, 0, c)+ \mathcal{O}({a-n},0,{c+n}) \]
if $c$ is even, while if $c$ is odd we have
\[
\mathcal{O}({a-n+1},0, {c-1}) \mathcal{O}({n-1},0,1) = \mathcal{O}(a,0, c)+\lambda\mathcal{O}({a-n+2}, 0, {c+n-2}).
\]
In both cases, the second invariant on the right-hand side belongs to Case 1 because $c \geq n+1$ and oth $n$ and $c$ are odd..\\
\underline{\bf Case 5}: $0<b<\frac{n}{2}$ then  $a+c>4n$, so that one cannot have both $a < 2n$ and $c < 2n$.  Without loss of generality we assume that $a >2n$ and $c < 2n$, and we write $a = 2n +p$ for $0<p <2n$ and $c = 2n-q$ for $0 < q < 2n$. Since $a + c > 4n$ we have $p-q>0$.  We consider two cases.\\
\underline{\bf Case 5.1}: $a=2n+p>3n$. Then we can assume $a < 4n$ by Case 1, and  then $n\leq c <2n$. Therefore 
$$\mathcal{O}({2n+p},b,{n-q})\mathcal{O}(0,0,n) =\mathcal{O}({2n+p},b,{2n-q}) +$$
\[
\begin{cases}
\mathcal{O}({3n+p-1},{b+1}, {n-q-1})\quad\mathrm{if}\;q\;\mathrm{even}\\
\mathcal{O}({3n+p+1},{b-1},{n-q+1}) \quad\mathrm{if}\;q\;\mathrm{odd}.\\
\end{cases}
\]
Since $3n+p+1 > 3n+p-1 = 2n + p + n-1> 3n +n-1 = 4n-1$, and both cases fall into Case 2.\\
\underline{\bf Case 5.2}: $a = 2n+p \leq 3n$. Then $(2b+2n-q) \geq 2n$. Therefore if $q$ is even:
$$
\mathcal{O}({2n+p},{b-\frac{q}{2}},0) \mathcal{O}(0,\frac{q}{2},{2n-q})= \mathcal{O}(a,b,c) +\lambda^{-\frac{q}{2}}\mathcal{O}(u^{2n+p} (du)^{b-\frac{q}{2}} (ud)^\frac{q}{2} u^{2n-q}).$$
If $q$ is odd:
$\mathcal{O}({2n+p},{b-\frac{q+1}{2}},1)\mathcal{O}(1, \frac{q-1}{2},{2n-q})=$$$
\mathcal{O}(a,b,c)+ \lambda^{-(q+1)/2}\mathcal{O}(u^{2n+p} (du)^{b-\frac{q+1}{2}} (ud)^\frac{q-1}{2} u^{2n-q} d^2).$$
Without complete simplification, it is clear that the degree of $u$ in the second
invariant on the right-hand side is always greater or equal to $(2n+p+2n-q)>4n$. This leads the invariant to fall into Case 2, completing Case 5.\\
\underline{\bf Case 6}: $\frac{n}{2}<b<n$.
 ~If $c$ is even we have
$\mathcal{O}({a},{b-\frac{n+1}{2}}, {c+1})\mathcal{O}(1,\frac{n-1}{2},0)$ $$ = \mathcal{O}(a,b,c) +\lambda^{(n-1)/2}\mathcal{O}({a+2b-n},{n-b-1}, {c+2b-n+2}).$$ 
If $c$ is odd we have:
$\mathcal{O}({a},{b-\frac{n-1}{2}}, {c-1})\mathcal{O}(0,\frac{n-1}{2}1)$ $$ = \mathcal{O}(a,b,c)+\lambda^{(n+1)/2}\mathcal{O}({a+2b-n+2},{n-b-1},{c+2b-n}).$$ 
Note that in both cases above, in the second invariant on the right-hand side, the degree of $du$, which is $n-b-1$, is smaller than $\frac{n}{2}$. so this case follows from Case 5.\\
\underline{\bf Case 7}: $n\leq b<2n$. Then in both the cases $c$ even and $c$ odd, we have:
$$\mathcal{O}(a,{b-n},{c})\mathcal{O}(0,{n},0) = \mathcal{O}(a,b,c)+\mathcal{O}({a+2b-2n+1},{2n-b-1}, {c+2b-2n+1}).$$
Note that the degree of $du$ is now $2n-b-1$, and $0\leq (2n-b-1) < n$. Thus the case falls into the union set of Cases 4, 5, and 6.
\end{proof}

Having shown we can generate the invariants using invariants of degree $\leq 4n$ we
now improve this bound, and together with Proposition \ref{dulowerboundnodd}, we compute $\beta(g)$.
\begin{theorem} \label{downupnodd} If $n \equiv 1 \mod{2}$ then $\beta(g) = 3n$.
\end{theorem}
\begin{proof}
It suffices to show that an invariant $(4n-2b-c, b, c)$  of degree $4n$ can be generated by invariants
of degree $\leq 3n$.  We first compute some special cases.    We have the following system of equations
$$\mathcal{O}(2n,0,0)^2 = \mathcal{O}(4n,0,0) + \mathcal{O}(2n,0,2n)$$
$$\mathcal{O}(n, 0, 0)\mathcal{O}(3n,0,0) = \mathcal{O}(4n,0,0) + \mathcal{O}(3n-1,1,n-1)$$
$$\mathcal{O}(2n,0,n) \mathcal{O}(n,0,0) = \mathcal{O}(3n-1,1,n-1) + \mathcal{O}(2n,0,0)$$
that can be solved, so in particular we can generate $\mathcal{O}(4n, 0, 0)$ and $\mathcal{O}(2n, 0, 2n)$
using invariants of lower degree.

We consider the orbit sums $\mO(a, b, c)$. where $a +2b+c = 4n$ and show they
can be generated by orbit sums of smaller degree.\\
\underline{\bf Case 1}: $b$ and $c$ have the same parity.
Consider the equation:
\begin{align*} \mathcal{O}(2n-2i-j,i,j)^2 & =  \lambda^{i+j}\mathcal{O}(2n,0,2n) \\&  + \begin{cases}
\mathcal{O}(4n-4i-2j,2i,2j) \quad \text{ if } j \text{ even}\\
\mathcal{O}(4n-4i-2j-1, 2i+1, 2j-1) \quad \text{ if } j \text{ odd}.\\
\end{cases} 
\end{align*}
Hence we can generate all invariants of degree $4n$ by lower degree invariants when, in each of the two summands, the $du$ and $d$ exponents have the same parity.\\
\underline{\bf Case 2}: $b$ and $c$ have  opposite parity.  As in the proof of  Proposition \ref{4nbound} we will consider cases according to the value of $b$.\\
    \underline{\bf Case 2.1}: 
$b=0$ and $c$ is odd. We can assume $c<2n$, then
\[
\mO(3n-c+1,0,c-1)\mO(n-1,0,1)=\mO(4n-c,0,c)+\lambda\mO(3n-c+2,0,n+c-2),
\]
since $n+c-2$ is even the second summand on the right-hand side of the equation is covered by Case 1. Hence
we can generate orbit sums of the form $\mathcal{O}(4n-c, 0, c)$ by lower degree invariants.\\
\underline{\bf Case 2.2}: $0 < b< n/2$. Without loss of generality we can assume $a < 2n$ and $2b + c = 4n-a > 2n$.  Then
$b+c > 2n-b > 2n-n/2 > n$.  When $b$ is  odd and $c$ is even we have:
$$\mathcal{O}(a, (b-1)/2,b+c+1-n)\mathcal{O}(1,(b-1)/2,n-b)= $$
$$ \mathcal{O}(a,b,c) + \lambda^{-(b+1)/2}\mathcal{O}(a+n-1, 0 , c+2b-n+1),$$
while if  $b$ is even and $c$ is odd we have:
$$\mathcal{O}(a,b/2,c+b-n)\mathcal{O}(0,b/2, n-b) =$$
$$\mathcal{O}(a,b,c) + \lambda^{-(b+1)/2}\mathcal{O}(a + n,0, c+ 2b-n).$$
In either case, the second summand on the right-hand side of the equation is covered by the Case 2.1.\\
\underline{\bf Case 2.3}: $n/2 < b <n$. This case follows   from the equations in Proposition \ref{4nbound} Case 6, for when $b$ is odd and $c$ is even, the second summand on the right-hand side of the equation is $\mathcal{O}(a+2b-n, n-b-1, c+2b-n+2)$, and the second and third components are both odd. Similarly when $b$ is even and $c$ is odd.
 In both cases the second summand is generated by lower degree invariants by Case 1.\\
 \underline{\bf Case 2.4}: $n \leq b <2n$. This case follows from the equation in Proposition \ref{4nbound} Case 7, for the second summand on the right-hand side of the equation is
 $\mathcal{O}(a+2b-2n+1, 2n-b-1, c+2b-2n +1)$ and $0 \leq 2n-b-1 <n$ with the second
 and third components of opposite parity, so generated by lower degree invariants
 by the union of Cases 2.1, 2.2, and 2.3.
\end{proof}

\begin{question}
It would be interesting to find $\beta(g)$
for $A(0,-1)$ and $A(2,-1)$. For $n=1$, one 
 can show that when $g$ acts on $A(0,-1)$ the 
 invariant $(du)^2 + (ud)^2$ of  degree $4$ is needed to generate $A^G$, and computer calculations suggest that $\beta(g) = 4$; for $n$ odd is $\beta(g) = 4n$?  For $n=1$ and $g$ acting on $A(2,-1)$ computer calculations suggest $\beta(g) = 2$. For noncommutative algebras how is $\beta(g)$ related to the order of the group?  
\end{question}

\bibliography{invariantproject}

\providecommand{\bysame}{\leavevmode\hbox to3em{\hrulefill}\thinspace}
\providecommand{\MR}{\relax\ifhmode\unskip\space\fi MR }
\providecommand{\MRhref}[2]{%
  \href{http://www.ams.org/mathscinet-getitem?mr=#1}{#2}
}
\providecommand{\href}[2]{#2}
\begin{thebibliography}{10}

\bibitem{AS}
M.~Artin and W.~F. Schelter, \emph{Graded algebras of global dimension {$3$}},
  Adv. in Math. \textbf{66} (1987), no.~2, 171--216.

\bibitem{BR}
G.~Benkart and T.~Roby, \emph{Down-up algebras}, J. Algebra \textbf{209}
  (1998), no.~1, 305--344.

\bibitem{DK}
H.~Derksen and G.~Kemper, \emph{Computational {I}nvariant {T}heory},
  {E}ncyclopedia of {M}athematical {S}ciences 130: {I}nvariant {T}heory and
  {A}lgebraic {T}ransformation {G}roups {I}, Springer-Verlag, Berlin, 2002.

\bibitem{DS}
H.~Derksen and J.~Sidman, \emph{Castelnuovo-{M}umford regularity by
  approximation}, Adv. Math. \textbf{188} (2004), no.~1, 104--123.

\bibitem{DH}
M.~Domokos and P.~Heged\"{u}s, \emph{Noether's bound for polynomial invariants
  of finite groups}, Arch. Math. \textbf{74} (2000), 161--167.

\bibitem{Fl}
P.~Fleischmann, \emph{The {N}oether bound in invariant theory of finite
  groups}, Adv. Math. \textbf{156} (2000), no.~1, 23--32.

\bibitem{Fo}
J.~Fogarty, \emph{On {N}oether's bound for polynomial invariants of a finite
  group}, Electron. Res. Announc. Amer. Math. Soc. \textbf{7} (2001), 5--7.

\bibitem{G}
F.~Gandini, \emph{Ideals of subspace arrangements}, Ph.D. {T}hesis, University
  of Michigan, Ann Arbor, May 2019, ORCID iD: 0000-0002-2619-3555.

\bibitem{M2}
Daniel~R. Grayson and Michael~E. Stillman, \emph{Macaulay2, a software system
  for research in algebraic geometry}, Available at
  \url{http://www.math.uiuc.edu/Macaulay2/}.

\bibitem{KK}
E.~Kirkman and J.~Kuzmanovich, \emph{Fixed subrings of {N}oetherian graded
  regular rings}, J. Algebra \textbf{288} (2005), no.~2, 463--484.

\bibitem{KKZ}
E.~Kirkman, J.~Kuzmanovich, and J.~Zhang, \emph{Invariants of (-1)-skew
  polynomial rings under permutation representations}, Recent advances in
  representation theory, quantum groups, algebraic geometry, and related
  topics, 155--192, Contemp. {M}ath., vol. 623, Amer. Math. Soc., Providence,
  RI, 2014.

\bibitem{KMP}
E.~Kirkman, I.~M. Musson, and D.~S. Passman, \emph{Noetherian down-up
  algebras}, Proc. Amer. Math. Soc. \textbf{127} (1999), no.~11, 3161--3167.

\bibitem{Neusel}
M.~Neusel, \emph{Degree bounds: An invitation to postmodern invariant theory},
  Topology and its Applications \textbf{154} (2007), no.~4, 792--814.

\bibitem{N}
E.~Noether, \emph{Der endlichkeitssatz der invarianten endlicher gruppen},
  Math. Ann \textbf{77} (1916), 89--92.

\bibitem{Se}
M.~Sezer, \emph{Sharpening the generalized {N}oether bound in the invariant
  theory of finite groups}, J. Algebra \textbf{254} (2002), no.~2, 252--263.

\bibitem{S}
P.~Symonds, \emph{On the {C}astelnuovo-{M}umford regularity of rings of
  polynomial invariants}, Annals of Math. \textbf{174} (2011), 499--517.

\end{thebibliography}
\bibliographystyle{amsplain}




\end{document}